\documentclass[a4paper,11pt]{amsart}

\usepackage{amsfonts,latexsym,rawfonts,amsmath,amssymb,amsthm, mathrsfs, lscape}
\usepackage[all]{xy}
\usepackage[english]{babel}
\usepackage[utf8]{inputenc}

\usepackage{graphicx}

\usepackage{xcolor}

\usepackage{tikz-cd}

\usepackage{array, tabularx}

\usepackage{setspace}
\usepackage{textcomp}

\usepackage{dynkin-diagrams}

\usepackage[hypertexnames=false,
backref=page,
    pdftex,
    pdfpagemode=UseNone,
    breaklinks=true,
    extension=pdf,
    colorlinks=true,
    linkcolor=blue,
    citecolor=blue,
    urlcolor=blue,
]{hyperref}

\usepackage[top=1.5in, bottom=.9in, left=0.9in, right=0.9in]{geometry}

\newcommand{\referenza}{}

\newtheorem{thm}{Theorem}[section]
\newtheorem*{thm*}{Theorem \referenza}

\newtheorem*{cor*}{Corollary \referenza}
\newtheorem{lem}[thm]{Lemma}
\newtheorem*{lem*}{Lemma \referenza}

\newtheorem*{prop*}{Proposition \referenza}

\newtheorem*{conj*}{Conjecture \referenza}
\newtheorem{rmk}[thm]{Remark}
\newtheorem*{rmk*}{Remark}

\newtheorem{defi}[thm]{Definition}

\numberwithin{equation}{section}

\def \H {\mathcal H}

\def \R {\mathbb R}

\def \C {\mathbb C}
\def \Z {\mathbb Z}

\DeclareMathOperator{\im}{i}

\newcommand{\p}{\partial}
\renewcommand{\bar}{\overline}

\allowdisplaybreaks[1]

\title[Bott--Chern cohomology of SU$(3)$, Spin$(5)$ and G$_2$ and the Pluriclosed flow]{Global stability of the Pluriclosed flow on compact simply-connected simple Lie groups of rank two}

\author{Giuseppe Barbaro}
\address[Giuseppe Barbaro]{Università La Sapienza, Roma}
\email{g.barbaro@uniroma1.it}

\address{Dipartimento di Matematica ``Guido Castelnuovo", Università la Sapienza, Piazzale Aldo Moro, 5, 00185 Roma, Italy} 
\thanks{The author is supported by GNSAGA of INdAM and by project PRIN2017 ``Real and Complex Manifolds: Topology, Geometry and holomorphic dynamics'' (code 2017JZ2SW5)}

\begin{document}

\begin{abstract}
    We compute the $(1,1)$-Aeppli cohomology of compact simply-connected simple Lie groups of rank two. In particular, we verify that they are of dimension one and generated by the classes of the Bismut flat metrics coming from the Killing forms. This yields a result on the stability of the pluriclosed flow on these manifolds. Moreover, we show that for compact simply-connected simple Lie groups of rank two the Dolbeaut cohomology, as well as the Bott--Chern and the Aeppli cohomologies, arise from just the left-invariant forms and we computed the whole Bott--Chern diamonds of SU$(3)$ and Spin$(5)$ when they are equipped with a left-invariant isotropic complex structure.
\end{abstract}

\maketitle

\section{Introduction}
    Given a Hermitian manifold $(X,J,g)$, the \emph{Bismut connection} $\nabla^{B}$ associated to $(g,J)$ is a Hermitian connection on $X$ with totally skew-symmetric torsion, where by Hermitian connection we mean a connection which is compatible with both the metric and the complex structure, i.e. $\nabla g = \nabla J =0$. It is described with respect to the Levi-Civita connection $\nabla^{LC}$ as, 
    \begin{equation*}
        g\left(\nabla^{B}_{x}y,z\right)=g\left(\nabla^{LC}_{x}y,z\right)+\frac{1}{2}Jd\omega(x,y,z) \;,
    \end{equation*}
    where $\omega$ is the K{\"a}hler form associated to $g$ and $J$ acts as $Jd\omega(\cdot,\cdot,\cdot)=-d\omega(J\cdot,J\cdot,J\cdot)$. Since the Levi-Civita connection is torsion free, the above formula is prescribing the torsion of this connection as 
    $$ T^B (x,y,z)=Jd\omega(x,y,z) \;. $$
    We indicate by $Ric^B$ the Ricci curvature tensor of the Bismut connection, which is $Ric^B (X,Y) := tr_g\Omega^B(X,Y)$, i.e. the contraction of the endomorphism part of the Bismut curvature tensor $\Omega^B(X,Y)=\left[\nabla_X,\nabla_Y\right] - \nabla_{[X,Y]}$. We also use $Ric^B$ to indicate the Ricci form of the Bismut connection which satisfies the following equation (see for example \cite{Me} and the references therein),
    $$ Ric^B(g) = -\sqrt{-1}\p\bar\p\log\omega^n - dd_g^*\omega \;.$$
    In the equality above, $d^*_g=\p^*_g+\bar\p^*_g$ where $\partial^*_g:\land^{p+1,q}X\rightarrow\land^{p,q}X$ and $\overline{\partial}^*_g:\land^{p,q+1}X\rightarrow\land^{p,q}X$ are the $L^2_g$-adjoint operators of $\partial$ and $\overline{\partial}$ respectively. In particular, the $(1,1)$-component of this form is
    \begin{equation}\label{eq: Ricci Bismut 1,1}
        \left(Ric^B(g)\right)^{1,1} = -\sqrt{-1}\p\bar\p\log\omega^n - \left(\partial\partial^*_g \omega +\overline{\partial\partial}^*_g \omega\right) \;.
    \end{equation}
    
    \medskip
    
    In \cite{ST10} Streets and Tian introduced the {\em pluriclosed flow} with the intent to study the complex geometry of Hermitian non-K{\"a}hler manifolds. This is a parabolic flow of Hermitian metrics defined as follows. On a Hermitian manifold $(X,g_0,J)$ it evolves the metric obeying the equation:
    \begin{equation*}
        \begin{cases}
            \frac{\p }{\p t}\,g = - S + Q\\
            g(0)=g_0
        \end{cases}
    \end{equation*}
    where $S$ is the trace in the first two entries of the Chern curvature tensor $\Omega^{Ch}$ and $Q$ is a given quadratic polynomial in the torsion $T^{Ch}$ of the Chern connection. More precisely,
    \begin{equation*}
        S_{i\overline{j}}=(Tr_{g}\Omega^{Ch})_{i\overline{j}}=g^{k\overline{l}}\Omega^{Ch}_{k\overline{l}i\overline{j}} \;,
    \end{equation*}
    and
    \begin{align*}
	    Q_{i\overline{j}}=g^{k\overline{l}}g^{m\overline{n}}T^{Ch}_{ik\overline{n}}T^{Ch}_{\overline{jl}m} \;.
    \end{align*}
    
    In the literature, a Hermitian metric is said to be {\em pluriclosed} or {\em SKT} when its associated K{\"a}hler form $\omega$ satisfies $dd^c \omega=0$. It can be proved that the pluriclosed flow preserves this condition whenever it occurs. Indeed, it evolves a pluriclosed metric in the direction of the $(1,1)$-component of its Bismut Ricci form, which is $dd^c$-closed (see (\ref{eq: Ricci Bismut 1,1})). More precisely, given a complex manifold $(X,J)$ together with a pluriclosed metric $\omega_0$ the pluriclosed flow evolves as
    \begin{equation*}
        \begin{cases}
            \frac{\p }{\p t}\omega = - \left(Ric^B(\omega)\right)^{1,1}\\
            \omega(0)=\omega_0
        \end{cases}
    \end{equation*}
    Moreover, in \cite[Theorem 9]{AOUV} it is proved that metrics with vanishing Bismut curvature, called {\em Bismut flat metrics}, are pluriclosed. Thus, clearly, Bismut flat metrics are fixed points for the pluriclosed flow.
    
    \medskip
    
    Recently, in \cite{GJS} the authors implemented a beautiful machinery based on {\em Generalized Geometry} to compare metrics with the same torsion class, which, for a generic metric $\omega$, is the class of $[\p\omega]\in H^{2,1}_{\bar\p}$. Thanks to it, they proved that the Bismut flat metrics are attractive for the pluriclosed flow in their torsion class. Precisely, they proved the following theorem.
    \begin{thm}[Theorem 1.2 of \cite{GJS}]\label{thm: GJS convergence}
        Let $(X^{2n},J, \omega_{BF})$ be a compact Bismut flat manifold. Given $\omega_0$ a pluriclosed metric such that $[\p\omega_0] = [\p\omega_{BF} ] \in H^{2,1}_{\bar\p}(X)$, the solution to pluriclosed flow with initial data $\omega_0$ exists on $[0, \infty)$ and converges to a Bismut flat metric $\omega_\infty$.
    \end{thm}
    It is conjectured that the SKT metrics with vanishing Bismut Ricci form should have a similar behaviour, however, this is the first result showing that a natural class of non-K{\"a}hler metrics is attractive for this flow.
    
    \medskip
    
    We fix here the notation for the {\em Bott--Chern} and {\em Aeppli cohomologies}. The Bott--Chern cohomology, \cite{bott-chern}, is the $\Z^2$-graded algebra
    $$ H^{\bullet,\bullet}_{BC}(X) \;:=\; \frac{\ker\p\cap\ker\bar\p}{\mbox{Im }\p\bar\p} \;, $$
    while the Aeppli cohomology, \cite{aeppli}, is the $\Z^2$-graded $H_{BC}(X)$-module
    $$ H^{\bullet,\bullet}_{A}(X) \;:=\; \frac{\ker\p\bar\p}{\mbox{Im }\p + \mbox{Im }\bar\p} \;, $$
    both with respect to the double complex $\left(\wedge^{\bullet,\bullet}X,\, \p,\, \bar\p\right)$.
    
    \medskip
    
    Theorem \ref{thm: GJS convergence} together with the knowledge of the $(1,1)$-Aeppli cohomology of the manifold lead to global stability results. For example, the Hopf surface U$(1)\times \mbox{SU}(2)$ and the Calabi--Eckman manifold $\mbox{SU}(2)\times \mbox{SU}(2)$ are known to be Bismut flat when they are equipped with their standard complex structures and the metrics $\omega_{BF}$'s coming from the Killing forms; moreover, in both cases the $(1,1)$-Aeppli cohomology is one-dimensional and generated by the class of $\omega_{BF}$, see Theorem 3.3 and Proposition 3.4 of \cite{AT15}. Given any pluriclosed metric $\omega$ on them, integrating over the fiber $\mathbb{T}^2$ we see that $[\omega]\neq 0$ in $H^{1,1}_A(X)$, and hence there exists a positive $\lambda$ such that $[\omega] = \lambda[\omega_{BF}]$ in $H^{1,1}_A(X)$. Using now that $[\p\omega] \in H^{2,1}_{\bar\p}(X)$ is the outcome of the natural map
    $$ H^{1,1}_A \xrightarrow{\p} H^{2,1}_{\bar\p}: [\omega] \mapsto [\p\omega]\;,$$
    we have that Theorem \ref{thm: GJS convergence} applies giving convergence of the pluriclosed flow to a Bismut-flat structure for any initial pluriclosed data.
    
    \medskip
    
    In this note we produce new examples of global stability for the pluriclosed flow other than on Calabi--Eckmann manifolds. We consider compact simply-connected simple Lie groups of rank two equipped with their isotropic complex structures (as defined in Section \ref{sec: isotrop cplx structure}) and the Hermitian metrics coming from the Killing forms $\omega_{BF}$'s. These are classified as SU$(3)$, Spin$(5)$ and G$_2$ (see Remark \ref{rmk: Pittie's presentation}). Moreover, they are Bismut flat and we prove that their $(1,1)$-Aeppli cohomologies are 
    $$ H^{1,1}_A \left(\mbox{SU}(3)\right) = \C\left<[\omega_{BF}]\right> \;; H^{1,1}_A \left(\mbox{Spin}(5)\right) = \C\left<[\omega_{BF}]\right> \;; H^{1,1}_A \left(\mbox{G}_2\right) = \C\left<[\omega_{BF}]\right> \;. $$
    We performed these computations combining the results of \cite{AK17} with the presentation of the Dolbeaut cohomology for compact simply-connected simple Lie group of rank two given in \cite{Pit}:
    \begin{equation}\label{eq: model cohom}
        H_{\bar\p}^{\bullet,\bullet}\left(X^{2n}\right) \;\simeq\; 
        \begin{cases}
            \left. \C\left[y_{1,1}\right] \middle\slash \left(\left(y_{1,1}\right)^{n-1}\right) \right. \otimes \wedge^{\bullet,\bullet} \left( \C\left\langle [u_{2,1}] \right\rangle \oplus \C\left\langle [x_{0,1}] \right\rangle \right) \; \text{  isotropic complex structure}\\
            \wedge^{\bullet,\bullet} \left( \C\left\langle [x_{0,1}] \right\rangle \oplus \C\left\langle [y_{1,1}] \right\rangle \oplus \C\left\langle [u_{n,n-1}] \right\rangle \right) \hspace{3.5em}\text{non-isotropic complex structure}
        \end{cases}
    \end{equation} 
    where subscripts denote bi-degree of the generators $x,y,u$ and the isotropic complex structures are the ones that together with the Killing metric give a Hermitian structure on $X$. We thus prove the following result.
    \begin{thm*}[Theorems \ref{thm: mean theorem1}, \ref{thm: mean theorem2} \& \ref{thm: mean theorem3}]
        Given a compact simply-connected simple Lie group of rank two $X$, consider the Bismut flat Hermitian structure $(J,\omega_{BF})$ coming from the Killing form. Then for any pluriclosed metric $\omega_0$ on $\left(X,J\right)$ there exists a positive $\lambda$ such that the solution to the pluriclosed flow with initial data $\omega_0$ exists on $[0, \infty)$ and converges to $\lambda\,\omega_{BF}$. 
    \end{thm*}
    
    For completeness, we compute the whole Hodge diamond for the Bott--Chern cohomology of SU$(3)$ and Spin$(5)$ when they are equipped with their isotropic complex structures (in Section \ref{sec: computations}). In particular, we prove that for compact simply-connected simple Lie groups of rank two the whole Dolbeaut, Bott--Chern and Aeppli cohomologies arise from just the left-invariant classes and we have the following Bott--Chern numbers, respectively for SU$(3)$ and Spin$(5)$:
    $$
    \begin{array}{ccccccccc}
        & & & & 1 & & & & \\
        & & & 0 &  & 0 & & & \\
        & & 0 & & 2 & & 0 & & \\
        & 0 & & 1 & & 1 & & 0 & \\
        0 & & 0 & & 2 & & 0 & & 0 \\
        & 0 & & 1 & & 1 & & 0 & \\
        & & 0 & & 1 & & 0 & &   \\
        & & & 1 & & 1 & & &  \\
        & & & & 1 & & & & \\
    \end{array} \; \quad 
    \text{  and  } \quad \; 
    \begin{array}{ccccccccccc}
        & & & & & 1 & & & & &  \\
        & & & & 0 & & 0 & & & & \\
        & & & 0 & & 2 & & 0 & & & \\
        & & 0 & & 1 & & 1 & & 0 & & \\
        & 0 & & 0 & & 2 & & 0 & & 0 & \\
        0 & & 0 & & 1 & & 1 & & 0 & & 0 \\
        & 0 & & 0 & & 2 & & 0 & & 0 & \\
        & & 0 &  & 1 & & 1 & & 0 & & \\
        & & & 0 & & 1 & & 0 & & & \\
        & & & & 1 & & 1 & & & & \\
        & & & & & 1 & & & & & 
    \end{array} 
    $$
    A direct computation can also be performed to compute the Bott--Chern numbers of G$_2$.

\section*{Acknowledgements}
I would like to thank Mario Garcia--Fernandez for many useful suggestions on the topics of this note. I am very thankful to him, the ICMAT and the UAM for their warm hospitality and to all people there for the stimulant environment. I am also grateful to Daniele Angella and to Francesco Pediconi for the help provided during this work. Many thanks also to the anonymous Referees for their useful comments.

\section{Complex structures on Bismut flat manifolds}
     In \cite{Sam} Samelson showed (by an explicit construction) that any even-dimensional compact Lie group admits a left-invariant complex structure compatible with the bi-invariant metric coming from the Killing form. 
     
     In \cite{AI} Alexandrov and Ivanov showed that any even dimensional connected Lie group equipped with a bi-invariant metric $g$ and a left-invariant complex structure which is compatible with $g$ is Bismut flat. Afterwords, in \cite{WYZ} the authors showed that up to taking the universal cover, these are the only existing Bismut flat manifolds. In other words, simply-connected Bismut flat manifolds have been characterized as {\em Samelson spaces}, which definition is as follows.
    \begin{defi}[\cite{WYZ}]
        A Samelson space is a Hermitian manifold $(G, g, J)$, where $G$ is a connected and simply-connected, even-dimensional Lie group, $g$ a bi-invariant metric on $G$, and $J$ a left-invariant complex structure on $G$ that is compatible with $g$.
    \end{defi}
    By Milnor's Lemma (Lemma 7.5 of \cite{Mil}), a simply-connected Lie group $G'$ with a bi-invariant metric must be the product of a compact semisimple Lie group with an additive vector group.
    \begin{lem}[Lemma 7.5 of \cite{Mil}]\label{lem: Milnor}
        Let $G$ be a simply-connected Lie group with a bi-invariant metric $\langle \cdot ,\cdot \rangle$. Then $G$ is isomorphic and isometric to the product $G_1 \times \cdots \times G_r \times {\mathbb R}^k$ where each $G_i$ is a simply-connected compact simple Lie group and ${\mathbb R}^k$ is the additive vector group with the flat metric.
    \end{lem}
    We can go even further in the classification of the Bismut flat manifolds; indeed, the compact simply-connected simple Lie groups are fully classified, they are:
    \begin{align*}
        A_n &= \mbox{SU}(n+1),\; n\geq 1  &&\dim(A_n) = n(n+2) &&& \mbox{rank}(A_n)=n;\\
        B_n &= \mbox{Spin} (2n+1),\; n\geq 2 &&\dim(B_n) = n(2n+1) &&& \mbox{rank}(B_n)=n;\\
        C_n &= \mbox{Sp}(2n),\; n\geq 3 &&\dim(C_n) = n(2n+1) &&& \mbox{rank}(C_n)=n;\\
        D_n &= \mbox{Spin} (2n),\; n\geq 4 &&\dim(D_n) = n(2n-1) &&& \mbox{rank}(D_n)=n;\\
        E_6 & &&\dim(E_6) = 78 &&& \mbox{rank}(E_6)=6;\\
        E_7 & &&\dim(E_7) = 133 &&& \mbox{rank}(E_7)=7;\\
        E_8 & &&\dim(E_8) = 248 &&& \mbox{rank}(E_8)=8;\\
        F_4 & &&\dim(F_4) = 52 &&& \mbox{rank}(F_4)=4;\\
        G_2 & &&\dim(G_2) = 14 &&& \mbox{rank}(G_2)=2.
    \end{align*}
    
    \begin{rmk}\label{rmk: Pittie's presentation}
        A$_2=\mbox{SU}(3)$, B$_2=\mbox{Spin}(5)$ and $\mbox{G}_2$ are the only compact simply-connected simple Lie groups of rank two. Hence, the Dolbeaut cohomology of these manifolds can be computed using Pittie's cohomology presentation (\ref{eq: model cohom}).
    \end{rmk}
    
    \medskip
    
    In \cite{Pit} Pittie gave a complete description of the moduli of left-invariant, integrable complex structures on even-dimensional compact Lie groups, proving that they all come from Samelson's construction in \cite{Sam}, namely, from a choice of a maximal torus, a complex structure on the Lie algebra of the torus, and a choice of positive roots for the Cartan decomposition. Let us now recall this construction with more details. 
    
    In the following let $G$ be an even dimensional connected Lie group equipped with a bi-invariant metric $\langle \cdot ,\cdot \rangle$ and denote by ${\mathfrak g}$ the Lie algebra of $G$ and by ${\mathfrak g}^\C $ its complexification. $\langle \cdot ,\cdot \rangle$ will also denote the induced inner product on ${\mathfrak g}$. 
    
    We know that the left-invariant complex structures on $G$ are linear maps $J: {\mathfrak g}  \rightarrow {\mathfrak g}$ such that $J^2=-I$ and
    \begin{equation*}
        J ( [X,Y] -[JX,JY] ) = [JX,Y] + [X, JY]
    \end{equation*}
    for any $X$, $Y$ in ${\mathfrak g}$.
    These are in one to one correspondence with the complex Lie subalgebras ${\mathfrak s} \subset {\mathfrak g}^\C$, such that $\langle {\mathfrak s}, {\mathfrak s} \rangle =0$, ${\mathfrak s} \cap {\mathfrak g}=0$, and ${\mathfrak s} \oplus \overline{{\mathfrak s}} = {\mathfrak g}^\C$. Such subspaces are called {\em Samelson subalgebras} of ${\mathfrak g}^\C$.
    
    Given $K$ a maximal torus of $G$, it is well-known that one has the $\mbox{ad}(K)$-invariant decomposition 
    \begin{equation*} 
        {\mathfrak g}^\C = {\mathfrak k}^\C \oplus \sum_{\alpha\in R^+} {\mathfrak g}_{\alpha} \oplus {\mathfrak g}_{-\alpha} ,
    \end{equation*}
    where ${\mathfrak k}^\C $ denotes the complexification of ${\mathfrak k}$, $R^+$ is the space of positive roots and
    $$ {\mathfrak g}_{\alpha} := \left\{v \in {\mathfrak g}^\C \; \big| \; [H, v] = \alpha(H)v \; \forall\; H \in {\mathfrak k}\right\}\;. $$
    
    Since $\dim(G)$ is even, we know that the abelian Lie algebra ${\mathfrak k}$ is even dimensional. So we can choose an almost complex structure on ${\mathfrak k}$ that is compatible with the metric. This is equivalent to choose a complex subalgebra ${\mathfrak a} \subset {\mathfrak k}^\C$ such that $\langle {\mathfrak a}, {\mathfrak a} \rangle =0$, ${\mathfrak a} \cap {\mathfrak k}=0$, and ${\mathfrak a} \oplus \overline{{\mathfrak a}} = {\mathfrak k}^\C$.
    Now one could simply take
    \begin{equation*}
        {\mathfrak s} = {\mathfrak a} \oplus \sum_{\alpha\in R^+} {\mathfrak g}_{\alpha}
    \end{equation*}
    to be the Samelson subalgebra. Pittie \cite{Pit} proved that, any left-invariant complex structure on $G$ is obtained this way. Moreover, he described the moduli space of left-invariant complex structures on $G$ as
    $$ m_2(G) = \left(GL(2k, \R)/GL(k, \C)\right) / F \;, $$
    where $k$ is the rank of $G$ and $F$ is a discrete group generated by the automorphisms of the abelian factor of $G$, the automorphisms of the Dynkin diagrams of the simple factors, and permutations among isomorphic simple factors. Here $m_2(G)$ is the space of left-invariant complex structures on $G$ up to automorphisms of $G$. Among these, the isotropic ones are given by the quotient by $F$ of $O(2k)/U(k)$.
    
    
    As a consequence, the moduli space of left-invariant complex structures of a compact simply-connected rank two simple Lie group $X$ is given by
    $$ m_2 (X) = \left(\H_+ \cup \H_-\right)/F, $$
    where $\H_\pm$ represent respectively the upper and lower half planes in $\C$; see the example in \cite[page 123]{Pit}. Indeed, in this case $GL(2,\R)/GL(1,\C) \simeq \H_+ \cup \H_-$. Moreover, $O(2)/U(1)\simeq\{\pm i\}$ as a subset of $\H_+ \cup \H_-$. For SU(3) we have that $F=\Z/2$, hence there is a unique isotropic left-invariant complex structure on it up to automorphisms. On the other hand, $Spin(5)$ and $G_2$ have two isotropic left-invariant complex structures up to automorphisms, since $F$ is trivial for them.

\section{The isotropic complex structure on compact simply-connected simple Lie groups of rank two}\label{sec: isotrop cplx structure}
    
    In this section we follow the Samelson construction to describe the left-invariant complex structures on compact simply-connected simple Lie groups of rank two. We are going to focus on the complex structures which are isotropic with respect to the Killing metric. In the following $\langle \cdot ,\cdot \rangle$ will denote the inner product on the algebra induced by the Killing form. Since the only compact simply-connected simple Lie groups of rank two are SU$(3)$, Spin$(5)$ and G$_2$ (see Remark \ref{rmk: Pittie's presentation}) we split this section in three parts. However, we may use the same symbols for objects which refer to different groups.
    
\subsection{SU$(3)$}\label{subsec: isotrop cplx structure SU}
    The group SU(3) is the group of $3\times 3$ unitary matrices with unit determinant,
    $$ \mbox{SU}(3):=\{U \in M(3,3;\C) \;|\; UU^{\dagger}=Id \;\wedge\; \det U = 1 \}\;.$$
    By differentiating this two conditions we get that its Lie algebra is made by the skew-Hermitian $(3\times3)$-matrices with zero trace, i.e.
    \begin{equation*}
        \begin{pmatrix}
            a\im & c + \im\, d & e + \im\,f\\
            -c + \im\,d & b\im & g + \im\,h\\
            -e + \im\,f & - g + \im\,h & -a\im - b\im
        \end{pmatrix}\;,
    \end{equation*}
    for real parameters $a,b,c,d,e,f,g,h$.
    We can take $\im$-times the Gell--Mann matrices as basis of the SU(3) Lie algebra $\mathfrak{su}(3)$:
    \begin{align*}
        &e^1= \begin{pmatrix}
                0 & \im & 0\\
                \im & 0 & 0\\
                0 & 0 & 0
             \end{pmatrix}\;
        &&e^2= \begin{pmatrix}
                0 & 1 & 0\\
                -1 & 0 & 0\\
                0 & 0 & 0
             \end{pmatrix}\;
        &&&e^3= \begin{pmatrix}
                \im & 0 & 0\\
                0 & -\im & 0\\
                0 & 0 & 0
             \end{pmatrix} \\
        &e^4= \begin{pmatrix}
                0 & 0 & \im\\
                0 & 0 & 0\\
                \im & 0 & 0
             \end{pmatrix}\;
        &&e^5= \begin{pmatrix}
                0 & 0 & 1\\
                0 & 0 & 0\\
                -1 & 0 & 0
             \end{pmatrix} &&& \\
        &e^6= \begin{pmatrix}
                0 & 0 & 0\\
                0 & 0 & \im\\
                0 & \im & 0
             \end{pmatrix} \;
        &&e^7= \begin{pmatrix}
                0 & 0 & 0\\
                0 & 0 & 1\\
                0 & -1 & 0
             \end{pmatrix}\;
        &&&e^8= \frac{1}{\sqrt{3}}\begin{pmatrix}
                \im & 0 & 0\\
                0 & \im & 0\\
                0 & 0 & -2\im
             \end{pmatrix}
    \end{align*} 
    In this way we have the global left-invariant frame $\left\{e^1,e^2,e^3,e^4,e^4,e^5,e^6,e^7,e^8\right\}$ on SU(3) with structure constants 
    $$ \left[e^i,e^j\right]=2\sum_{k=1}^8\lambda^{ijk}e^k $$  given by
    \begin{table}[h!]
        \begin{tabular}{l|lllllllll}
            ijk & 123 & 147 & 156  & 246 & 257 & 345 & 367  & 458         & 678         \\ \hline
            $\lambda^{ijk}$ & $-1$   & $-\frac{1}{2}$ & $\frac{1}{2}$ & $-\frac{1}{2}$ & $-\frac{1}{2}$ & $-\frac{1}{2}$ & $\frac{1}{2}$ & $-\frac{\sqrt{3}}{2}$ & $-\frac{\sqrt{3}}{2}$
        \end{tabular}
    \end{table}\\
    Here $\lambda^{ijk} = (-1)^{|\sigma|}\lambda^{\sigma(i,j,k)}$ for any permutation $\sigma$. With this notation 
    $$ \left<e^i,e^j\right>= 4\sum_{p,q=1}^8\lambda^{ipq}\lambda^{jpq}\;. $$
    Thus we notice that all the $e^i$ have the same norm.
    
    The maximal torus in SU(3) is given by the diagonal matrices and its Lie algebra $\mathfrak{k}\subset\mathfrak{su}(3)$ is generated by $e^3$ and $e^8$. The remaining six generators, outside the Cartan subalgebra, could be arranged into six roots. Consider $e^1\pm\im e^2$, $e^4\pm\im e^5$ and $e^6\pm\im e^7$; then we have the following relations:
    \begin{align*}
        &\left[e^3,e^1\pm\im e^2\right]=\pm2\im\left(e^1\pm\im e^2\right)
        &&\left[e^8,e^1\pm\im e^2\right]=0\\
        &\left[e^3,e^4\pm\im e^5\right]=\pm\im\left(e^4\pm\im e^5\right)
        &&\left[e^8,e^4\pm\im e^5\right]=\pm\sqrt{3}\im\left(e^4\pm\im e^5\right)\\
        &\left[e^3,e^6\pm\im e^7\right]=\mp\im\left(e^6\pm\im e^7\right)
        &&\left[e^8,e^6\pm\im e^7\right]=\pm\sqrt{3}\im\left(e^6\pm\im e^7\right)
    \end{align*}
    Therefore, $e^1+\im e^2$, $e^4+\im e^5$ and $e^6-\im e^7$ are three positive roots. We shall define 
    \begin{align*}
        &\varphi^1=e^1+\im e^2
        &&\varphi^3=e^6-\im e^7\\
        &\varphi^2=e^4+\im e^5
        &&\varphi^4=(1-a\im)e^3-b\im e^8\;\text{ with } a+\im b \in \H_-
    \end{align*}
    so that $\left< \varphi^1, \varphi^2, \varphi^3, \varphi^4\right>$, generate the Samelson subalgebras on SU(3). Among these the only isotropic Samelson subalgebra is detected by the choice $Je^3 = e^8$ since $\left<e^3,e^8\right>=0$ and they have the same norm. It corresponds to $a+\im b = -\im$ and we indicate it as $J_{0,-1}$.

\subsection{Spin$(5)$}\label{subsec: isotrop cplx structure Sp}
    The group Spin$(5)$ is the double cover of SO$(5)$, hence they share the same Lie algebra, which is given by the $5\times5$ skew-symmetric matrices. We take the following generators as basis of the Spin$(5)$ Lie algebra $\mathfrak{spin}(5)$:
    \begin{align*}
        & e^1 = A_{1,2} && e^2 = A_{1,3} &&& e^3 = A_{2,3} &&&& e^4 = A_{1,4} &&&&& e^5 = A_{2,4} \\
        & e^6 = A_{3,4} && e^7 = A_{1,5} &&& e^8 = A_{2,5} &&&& e^9 = A_{3,5} &&&&& e^{10} = A_{4,5}
    \end{align*}
    where $A_{i,j}$ represents the $5\times5$ skew-symmetric matrix with $1$ in the $(i,j)$-position; more precisely, $\left(A_{i,j}\right)_{p,q}=\delta_{i,p}\delta_{j,q} - \delta_{i,q}\delta_{j,p}$. Using this notation we describe the structure constants of the global left-invariant frame $\left\{e^1,e^2,e^3,e^4,e^4,e^5,e^6,e^7,e^8,e^9,e^{10}\right\}$ on Spin$(5)$ as
    \begin{equation*}
        \left[A_{i,j},A_{m,n}\right] = \delta_{mj} A_{in} - \delta_{nj} A_{im} - \delta_{mi} A_{jn} + \delta_{ni} A_{jm}\;.
    \end{equation*}
    A maximal torus in SO$(5)$ is given by the block-diagonal matrices of the form
    \begin{equation*}
        \begin{pmatrix}
            B_1 & 0 & 0\\
            0 & 0 & 0\\
            0 & 0 & B_2
        \end{pmatrix}\;,
    \end{equation*}
    where $B_1,\,B_2 \in \mbox{SO}(2)$. Thus its Lie algebra $\mathfrak{k}\subset\mathfrak{spin}(5)$ is generated by $e^1$ and $e^{10}$. The remaining eight generators, outside the Cartan subalgebra, could be rearranged into eight roots. Consider the following relations:
    \begin{align*}
        &\left[e^1,e^2\pm\im e^3\right]=\mp\im\left(e^2\pm\im e^3\right)
        &&\left[e^{10},e^2\pm\im e^3\right]=0\\
        &\left[e^1,e^6\pm\im e^9\right]=0
        &&\left[e^{10},e^6\pm\im e^9\right]=\mp\im\left(e^6\pm\im e^9\right)
    \end{align*}
    \begin{align*}
        &\left[e^1,\left(e^4-e^8\right)\pm\im \left(e^7+e^5\right)\right]=\left[e^{10},\left(e^4-e^8\right)\pm\im \left(e^7+e^5\right)\right]=\mp\im\left(\left(e^4-e^8\right)\pm\im \left(e^7+e^5\right)\right)\\
        &\left[e^1,\left(e^4+e^8\right)\pm\im \left(e^7-e^5\right)\right]=-\left[e^{10},\left(e^4+e^8\right)\pm\im \left(e^7-e^5\right)\right]=\pm\im\left(\left(e^4+e^8\right)\pm\im \left(e^7-e^5\right)\right)
    \end{align*}
    Therefore, we have the eight roots $\pm(\im,0),\,\pm(0,\im),\,\pm(\im,\im),\,\pm(\im,-\im)$. A choice of positive roots leads to a Samelson subalgebra of Spin$(5)$ and we can take generators $\left< \varphi^1, \varphi^2, \varphi^3, \varphi^4, \varphi^5 \right>$ defined as
    \begin{align*}
        &\varphi^1 = (1-a\im) e^1 -b \im e^{10}\;\text{ with } a+\im b \in \H_+\cup\H_-
        &&\varphi^2 = e^2 +\im e^3
        &&&\varphi^3 = e^4 +\im e^5\\
        &
        && \varphi^4 = e^7 +\im e^8  &&&\varphi^5 = e^6 +\im e^9
    \end{align*}
    The two isotropic Samelson subalgebras are detected by the choices $a+\im b = \pm\im$ and we indicate the corresponding complex structures with $J_{\pm}$. Indeed, it can be verified that $\left<e^1,e^{10}\right>=0$ and they have the same norm.

\subsection{G$_2$}\label{subsec: isotrop cplx structure G2}
    The group G$_2$ is the simple exceptional Lie group of rank two. Its Lie algebra $\mathfrak{g}_2$ is described by the Dynkin diagram
    \begin{tikzpicture}
        \dynkin G2
    \end{tikzpicture}.
    Hence, we can fix a system of simple roots given by $\{\alpha_1,\alpha_2\}$ with Cartan matrix 
    \begin{equation*}
        \begin{pmatrix}
            2  & -3 \\
            -1 & 2 
        \end{pmatrix}\;,
    \end{equation*}
    Therefore, the positive roots are:
    $$ R^+ = \{\alpha_1,\alpha_2,\alpha_1+\alpha_2,2\alpha_1+\alpha_2,3\alpha_1+\alpha_2,3\alpha_1+2\alpha_2\}\;. $$
    We now construct a basis of $\mathfrak{g}_2^\C$ adapted to the roots system as follows. Take $\varphi_1\in\mathfrak{g}_{2,\alpha_1}, \varphi_2\in\mathfrak{g}_{2,\alpha_2}$ and $\bar\varphi_1\in\mathfrak{g}_{2,-\alpha_1}, \bar\varphi_2\in\mathfrak{g}_{2,-\alpha_2}$. Thanks to the relation $\left[\mathfrak{g}_{2,\alpha},\mathfrak{g}_{2,\beta}\right]\subset\mathfrak{g}_{2,\alpha+\beta}$, we define the eigenvectors of the others roots starting from $\varphi_1,\bar\varphi_1,\varphi_2,\bar\varphi_2$. Namely,
    \begin{align*}
        & \varphi_3 = [\varphi_1,\varphi_2] && \varphi_4 = [\varphi_1,\varphi_3]  &&& \varphi_5 = [\varphi_1,\varphi_4] &&&& \varphi_6 = [\varphi_2,\varphi_5] \\
        & \bar\varphi_3 = [\bar\varphi_1,\bar\varphi_2] && \bar\varphi_4 = [\bar\varphi_1,\bar\varphi_3] &&& \bar\varphi_5 = [\bar\varphi_1,\bar\varphi_4] &&&& \bar\varphi_6 = [\bar\varphi_2,\bar\varphi_5] 
    \end{align*}
    Moreover, the generators of the torus $\mathfrak{k}\in\mathfrak{g}_2$ can be chosen as $h_i = [\varphi_i,\bar\varphi_i]$ for $i = 1,2$.\\
    Notice that $\left[\mathfrak{g}_{2,\alpha},\mathfrak{g}_{2,\beta}\right] = 0$ if $\alpha+\beta$ is not a root. Moreover, using the Jacobi identity we can compute all the non vanishing products, which we summarize in the following Table \ref{tab: g2 algebra}.
    \begin{table}[h!]
        \begin{tabular}{l|llllllllllll}
            $[\cdot,\cdot]$ & $\varphi_1$ & $\bar\varphi_1$  & $\varphi_2$ & $\bar\varphi_2$ & $\varphi_3$ & $\bar\varphi_3$  & $\varphi_4$ & $\bar\varphi_4$ & $\varphi_5$ & $\bar\varphi_5$ & $\varphi_6$ & $\bar\varphi_6$ \\ \hline
            $h_1$ & $2\varphi_1$ & $-2\bar\varphi_1$ & $-3\varphi_2$ & $3\bar\varphi_2$ & $-\varphi_3$ & $\bar\varphi_3$ & $\varphi_4$ & $-\bar\varphi_4$ & $3\varphi_5$ & $-3\bar\varphi_5$ & $0$ & $0$ \\
            $h_2$  & $-\varphi_1$ & $\bar\varphi_1$ & $2\varphi_2$ & $-2\bar\varphi_2$ & $\varphi_3$ & $-\bar\varphi_3$ & $0$ & $0$ & $-\varphi_5$ & $\bar\varphi_5$ & $\varphi_6$ & $-\bar\varphi_6$ \\
            $\varphi_1$ & & $h_1$ & $\varphi_3$ & $0$ & $\varphi_4$ & $3\bar\varphi_2$ & $\varphi_5$ & $4\bar\varphi_3$ & $0$ & $3\bar\varphi_4$ & $0$ & $0$ \\
            $\bar\varphi_1$ & & & $0$ & $\bar\varphi_3$ & $3\varphi_2$ & $\bar\varphi_4$ & $4\varphi_3$ & $\bar\varphi_5$ & $3\varphi_4$ & $0$ & $0$ & $0$ \\
            $\varphi_2$ & & & & $h_2$ & $0$ & $-\bar\varphi_1$ & $0$ & $0$ & $\varphi_6$ & $0$ & $0$ & $\bar\varphi_5$ \\
            $\bar\varphi_2$ & & & & & $-\varphi_1$ & $0$ & $0$ & $0$ & $0$ & $\bar\varphi_6$ & $\varphi_5$ & $0$ \\
            $\varphi_3$ & & & & & & $-h_1$   & $-\varphi_6$ & $4\bar\varphi_1$ & $0$ & $0$ & $0$ & $3\bar\varphi_4$ \\
                        & & & & & & $- 3h_2$ & & & & & &\\
            $\bar\varphi_3$ & & & & & & & $4\varphi_1$ & $-\bar\varphi_6$ & $0$ & $0$ & $3\varphi_4$ & $0$ \\
            $\varphi_4$ & & & & & & & & $8h_1$ & $0$ & $-12\bar\varphi_1$ & $0$ & $12\bar\varphi_3$ \\
                        & & & & & & & & $+ 12h_2$& & & & \\
            $\bar\varphi_4$ & & & & & & & & & $-12\varphi_1$ & $0$ & $12\varphi_3$ & $0$ \\
            $\varphi_5$ & & & & & & & & & & $-36h_1$ & $0$ & $36\bar\varphi_2$ \\
                        & & & & & & & & & & $-36h_2$ & &\\
            $\bar\varphi_5$ & & & & & & & & & & & $36\varphi_2$ & $0$ \\
            $\varphi_6$ & & & & & & & & & & & & $36h_1 $ \\
                        & & & & & & & & & & & & $+ 72h_2$
        \end{tabular}
        \caption{Algebra structure of $\mathfrak{g}_2$}
        \label{tab: g2 algebra}
    \end{table}
    
    Finally, it remains to assign a complex structure on the torus to describe the Samelson subalgebras of $\mathfrak{g}_2$. Namely, these are generated by $\left< \varphi^1, \varphi^2, \varphi^3, \varphi^4, \varphi^5, \varphi^6, \varphi^7\right>$, where $\varphi_7 = (1-a\im)h_1 - b\im h_2$ with $a+\im b \in \H_+ \cup \H_-$. Since, $\|\alpha_1\|^2 = 3 \|\alpha_2\|^2$ and $\left<h_1,h_2\right>=-\frac{1}{2}\|h_1\|^2$ the two isotropic complex structures are given by $Jh_1 = \pm\sqrt{3}(h_1 + 2h_2)$, that is $$\varphi_7 = \left(1\mp\sqrt{3}\im\right)h_1 \mp 2\sqrt{3}\im h_2 \;.$$
    We indicate these complex structures with $J_{\pm}$.

\section{Bott Chern cohomology of compact simply-connected simple Lie groups of rank two}\label{sec: computations}
    We now compute case by case the $(1,1)$-Aeppli cohomology of the compact simply-connected simple Lie groups of rank two in order to apply Theorem \ref{thm: GJS convergence}.
    
    We will use the following notation: $\varphi^{i\bar j} = \varphi^i\wedge\bar\varphi^j$.

\subsection{SU$(3)$}\label{subsec: computations SU}
    We consider the Hermitian pluriclosed manifold $\left(\mbox{SU}(3),J_{0,-1},\omega_{BF}\right)$ where $J_{0,-1}$ is the isotropic left-invariant complex structure as given in Section \ref{subsec: isotrop cplx structure SU} and $\omega_{BF}$ represents the Hermitian metric coming from the Killing form; more precisely it is 
    \begin{equation}\label{eq: kill metric}
        \omega_{BF} := \frac{\im}{2} \sum_{k=1}^4 \varphi^{k}\wedge\bar\varphi^{ k} \;.
    \end{equation}
    
    By computing the complex structure equations, we obtain
    $$\left\{\begin{array}{rcl}
            \p \varphi^1 &=& -\im\, \varphi^{14} + \im\, \varphi^{23} \\
            \p \varphi^2 &=& -\frac{1}{2}\left(\sqrt{3} + \im \right) \varphi^{24} \\
            \p \varphi^3 &=& \frac{1}{2}\left(\sqrt{3} - \im \right) \varphi^{34} \\
            \p \varphi^4 &=& 0
           \end{array}\right.
     \;\text{ and }\; 
     \left\{\begin{array}{rcl}
            \bar\p \varphi^1 &=& -\im\, \varphi^{1\bar4} \\
            \bar\p \varphi^2 &=& \im\, \varphi^{1\bar3} + \frac{1}{2}\left(\sqrt{3} - \im \right) \varphi^{2\bar4} \\
            \bar\p \varphi^3 &=& - \im\, \varphi^{1\bar2} - \frac{1}{2}\left(\sqrt{3} + \im \right) \varphi^{3\bar4} \\
            \bar\p \varphi^4 &=& \im\, \varphi^{1\bar1} + \frac{1}{2}\left(-\sqrt{3} + \im \right) \varphi^{2\bar2} + \frac{1}{2}\left(\sqrt{3} + \im \right) \varphi^{3\bar3}
           \end{array}\right. 
    $$
    In \cite{Pit}, a model for the Dolbeault cohomology of compact simply-connected simple Lie groups of rank two is given, see (\ref{eq: model cohom}). In particular, when SU(3) is equipped with its isotropic left-invariant complex structure it holds
    \begin{equation}\label{eq: model Dolb cohom}
        H_{\bar\p}^{\bullet,\bullet}\left(\mbox{SU}(3)\right) \;\simeq\; \left. \C\left[y_{1,1}\right] \middle\slash \left(\left(y_{1,1}\right)^{3}\right) \right. \otimes \wedge^{\bullet,\bullet} \left( \C\left\langle [u_{2,1}] \right\rangle \oplus \C\left\langle [x_{0,1}] \right\rangle \right) \;,
    \end{equation}
    where subscripts denote bi-degree.
    By this presentation we recover that the Hodge numbers are
    $$
        \begin{array}{ccccccccc}
        & & & & 1 & & & & \\
        & & & 0 & & 1 & & & \\
        & & 0 & & 1 & & 0 & & \\
        & 0 & & 1 & & 1 & & 0 & \\
        0 & & 0 & & 2 & & 0 & & 0 \\
        & 0 & & 1 & & 1 & & 0 & \\
        & & 0 & & 1 & & 0 & &   \\
        & & & 1 & & 0 & & &  \\
        & & & & 1 & & & & \\
    \end{array} 
    $$
    Consider the sub-complex of left-invariant forms
    $$ \iota \colon \bigwedge \left\langle \varphi^1,\, \varphi^2,\, \varphi^3,\, \varphi^4,\, \bar\varphi^1,\, \bar\varphi^2,\, \bar\varphi^3,\, \bar\varphi^4 \right\rangle \hookrightarrow \wedge^{\bullet,\bullet}\mbox{SU}(3) \;. $$
    Since it is closed for the $\C$-linear Hodge-$*$-operator $*_{g_{BF}} \colon \wedge^{\bullet_1,\bullet_2}X \to \wedge^{n-\bullet_2,n-\bullet_1}X$ associated to $g_{BF}$, $H_{\bar\p}(\iota)$ is injective, see Theorem 1.6 and Remark 1.9 of \cite{AK17}. By knowing the Hodge numbers, we want to prove that it is also surjective. Thanks to structure of the model (\ref{eq: model Dolb cohom}) it is enough to check that the sub-complex cohomology has the same dimension as the complex in bi-degree $(0,1)$, $(1,1)$ and $(2,1)$. That is, to verify that the sub-complex has cohomologies $H^{0,1}_{\bar\p}(\mbox{SU}(3))_{inv}$, $H^{1,1}_{\bar\p}(\mbox{SU}(3))_{inv}$ and $H^{2,1}_{\bar\p}(\mbox{SU}(3))_{inv}$ of dimension one. We shall highlight that we chose the sub-complex of the left-invariant forms to exploit its algebraic structure. Indeed, the Pittie model induces the following sub-complex
    $$ \iota' \colon \bigwedge \left\langle x_{0,1},\, \bar{x_{0,1}},\, y_{1,1},\,\bar{y_{1,1}},\,u_{2,1},\,\bar{u_{2,1}} \right\rangle \hookrightarrow \wedge^{\bullet,\bullet}\mbox{SU}(3) \;,$$
    which, by construction, verifies $H_{\bar\p}(\iota')$ surjective. However, it is not trivial to prove that it closed for the $\C$-linear Hodge-$*$-operator associated to any metric, and thus it is difficult to prove that $H_{\bar\p}(\iota')$ is also injective.
    
    A direct computation leads to the following conditions
    \begin{align*}
        H^{0,1}_{\bar\p}(\mbox{SU}(3))_{inv} &= \C\left< \left[\bar\varphi^4\right] \right> \;;\\
        H^{1,1}_{\bar\p}(\mbox{SU}(3))_{inv} &= \C\left< \left[\varphi^{1\bar1}+\varphi^{2\bar2}\right] \right>\;;\\
        H^{2,1}_{\bar\p}(\mbox{SU}(3))_{inv} &= \C\left< \left[2\varphi^{14\bar1}-2\varphi^{23\bar1} + (1-\sqrt{3}\im)\varphi^{24\bar2} + (1+\sqrt{3}\im)\varphi^{34\bar3}\right] \right> \;,
    \end{align*}
    Therefore, $H_{\bar\p}(\iota)$ is an isomorphism and the formal representative $x_{0,1},\,y_{1,1}$ and $u_{2,1}$ of Pittie's model are respectively in the left-invariant classes $$\left[\bar\varphi^4\right],\,\left[\varphi^{1\bar1}+\varphi^{2\bar2}\right] \text{ and } \left[2\varphi^{14\bar1}-2\varphi^{23\bar1} + (1-\sqrt{3}\im)\varphi^{24\bar2} + (1+\sqrt{3}\im)\varphi^{34\bar3}\right]\;.$$
    In particular, we get that
    \begin{eqnarray*}
        H^{\bullet,\bullet}_{\bar\p}\left(\mbox{SU}(3)\right) &=& 
        \C\left\langle 1 \right\rangle 
        \oplus \C\left\langle \left[\varphi^{\bar{4}}\right] \right\rangle 
        \oplus \C\left\langle \left[ \varphi^{1\bar1}+\varphi^{2\bar2} \right] \right\rangle \\
        && \oplus\, \C\left\langle \left[ 2\varphi^{14\bar1}-2\varphi^{23\bar1} + (1-\sqrt{3}\im)\varphi^{24\bar2} + (1+\sqrt{3}\im)\varphi^{34\bar3} \right] \right\rangle 
        \oplus\, \C\left\langle \left[\varphi^{1\bar1\bar4}+\varphi^{2\bar2\bar4}\right] \right\rangle \\
        && \oplus \C\left\langle \left[2\varphi^{14\bar1\bar4}-2\varphi^{23\bar1\bar4} + (1-\sqrt{3}\im)\varphi^{24\bar2\bar4} + (1+\sqrt{3}\im)\varphi^{34\bar3\bar4}\right],\; \left[ \varphi^{12\bar1\bar2} \right] \right\rangle \\
        && \oplus \C\left\langle \left[(3-\sqrt{3}\im)\varphi^{124\bar1\bar2} + (1+\sqrt{3}\im) \varphi^{134\bar1\bar3} + (1+\sqrt{3}\im)\varphi^{234\bar2\bar3}\right] \right\rangle 
        \oplus \C\left\langle \left[\varphi^{12\bar1\bar2\bar4} \right] \right\rangle \\
        && \oplus \C\left\langle \left[(3-\sqrt{3}\im)\varphi^{124\bar1\bar2\bar4} + (1+\sqrt{3}\im) \varphi^{134\bar1\bar3\bar4} + (1+\sqrt{3}\im)\varphi^{234\bar2\bar3\bar4} \right] \right\rangle \\
        && \oplus \C\left\langle \left[\varphi^{1234\bar1\bar2\bar3}\right] \right\rangle \oplus \C\left\langle \left[\varphi^{1234\bar1\bar2\bar3\bar4}\right] \right\rangle \;.
    \end{eqnarray*}

    By Theorem 1.3 and Proposition 2.2 in \cite{AK17}, we also have that $H_{BC}(\iota)$ is an isomorphism. In particular, the Bott--Chern cohomology arises from just the left-invariant forms and a direct computation leads to 
    \begin{eqnarray*}
        H^{\bullet,\bullet}_{BC}\left(\mbox{SU}(3)\right) &=& \C\left\langle 1 \right\rangle \oplus \C\left\langle \left[\varphi^{1\bar1} + \varphi^{2\bar2} \right], \left[\varphi^{2\bar2} - \varphi^{3\bar3} \right] \right\rangle \\
        &&\oplus\, \C\left\langle \left[ 2\varphi^{14\bar1}-2\varphi^{23\bar1} + (1-\sqrt{3}\im)\varphi^{24\bar2} + (1+\sqrt{3}\im)\varphi^{34\bar3} \right] \right\rangle \\
        &&\oplus\, \C\left\langle \left[ 2\varphi^{1\bar1\bar4}-2\varphi^{1\bar2\bar3} + (1+\sqrt{3}\im)\varphi^{2\bar2\bar4} + (1-\sqrt{3}\im)\varphi^{3\bar3\bar4} \right] \right\rangle \oplus \C\left\langle \left[\varphi^{12\bar1\bar2}\right],\left[\varphi^{13\bar1\bar3}\right] \right\rangle \\
        && \oplus \C\left\langle \left[2\varphi^{124\bar1\bar2}+(\sqrt{3}\im-1)\varphi^{234\bar2\bar3}\right] \right\rangle \oplus \C\left\langle \left[2\varphi^{12\bar1\bar2\bar4}-(\sqrt{3}\im+1)\varphi^{23\bar2\bar3\bar4}\right] \right\rangle \\
        && \oplus \C\left\langle \left[ \varphi^{123\bar1\bar2\bar3}\right] \right\rangle \oplus \C\left\langle \left[\varphi^{1234\bar1\bar2\bar3}\right] \right\rangle \oplus \C\left\langle \left[\varphi^{123\bar1\bar2\bar3\bar4}\right] \right\rangle \oplus  \C\left\langle \left[\varphi^{1234\bar1\bar2\bar3\bar4}\right] \right\rangle \;.
    \end{eqnarray*}
    Hence, the Bott-Chern numbers are
    $$
        \begin{array}{ccccccccc}
        & & & & 1 & & & & \\
        & & & 0 & & 0 & & & \\
        & & 0 & & 2 & & 0 & & \\
        & 0 & & 1 & & 1 & & 0 & \\
        0 & & 0 & & 2 & & 0 & & 0 \\
        & 0 & & 1 & & 1 & & 0 & \\
        & & 0 & & 1 & & 0 & &   \\
        & & & 1 & & 1 & & &  \\
        & & & & 1 & & & & \\
    \end{array} 
    $$
    We know that for a Hermitian metric $g$ on a complex manifold $X$, the $\C$-linear Hodge-$*$-operator induces the isomorphism
    $$ *_{g_{BF}} \colon H^{\bullet_1,\bullet_2}_{BC}(X) \stackrel{\simeq}{\to} H^{n-\bullet_2,n-\bullet_1}_{A}(X) \;. $$
    Hence, as a corollary, we get that $\dim H^{1,1}_{A}\left(\mbox{SU}(3)\right) = 1$. In particular, we have that
    $$ H^{1,1}_A\left(\mbox{SU}(3)\right) = \C\left<\left[\varphi^{1\bar1}+\varphi^{2\bar2}+\varphi^{3\bar3}+\varphi^{4\bar4}\right]\right> = \C \left<\left[ \omega_{BF} \right]\right>\;, $$
    thus thanks to Theorem \ref{thm: GJS convergence} the following result holds.
    \begin{thm}\label{thm: mean theorem1}
        Consider $\left(\mbox{SU}(3),J_{0,-1},\omega_{BF}\right)$ where $\omega_{BF}$ is the metric given by the Killing form (as in (\ref{eq: kill metric})) and $J_{0,-1}$ is the isotropic left-invariant complex structure defined in Section \ref{subsec: isotrop cplx structure SU}. Given any pluriclosed metric $\omega_0$ on $\left(\mbox{SU}(3),J_{0,-1}\right)$ the solution to the pluriclosed flow with initial data $\omega_0$ exists on $[0, \infty)$ and converges to a Bismut flat metric $\omega_\infty$. In particular, there exists a positive $\lambda$ such that $[\omega_0] = \lambda [\omega_{BF}]$ in $H^{1,1}_A \left(\mbox{SU}(3)\right)$ and $\omega_\infty = \lambda\,\omega_{BF}$. 
    \end{thm}
    \begin{proof}
        First of all, we notice that the complex structure $J_{0,-1}$ is such that the maximal torus $\mathbb{T}^2$ in $\mbox{SU}(3)$ is a complex submanifold.
        Indeed, following the Samelson's construction, $J_{0,-1}$ has been defined firstly by choosing a complex structure on the maximal torus and then completing it to a complex structure of $\mbox{SU}(3)$.
        
        Now, given any pluriclosed metric $\omega_0$ on $\left(\mbox{SU}(3),J_{0,-1}\right)$, integrating over the maximal torus $\mathbb{T}^2$ we see that $[\omega_0]\neq 0$ in $H^{1,1}_A\left(\mbox{SU}(3)\right)$. Since the $(1,1)$-Aeppli cohomology of SU(3) is generated by $[\omega_{BF}]$, there exists a constant $\lambda$ such that $[\omega_0] = \lambda [\omega_{BF}]$ in $H^{1,1}_A \left(\mbox{SU}(3)\right)$, which must be positive cause the integrals of $\omega_0$ and $\omega_{BF}$ on $\mathbb{T}^2$ are both positive. Thus, Theorem \ref{thm: GJS convergence} applies ensuring long-time existence of the pluriclosed flow with initial data $\omega_0$ and convergence to a Bismut flat metric $\omega_\infty \in \lambda[\omega_{BF}]$.
        
        We have seen that Bismut flat metrics are bi-invariant and it is well known that any invariant symmetric bi-linear form on a simple Lie group must be a multiple of the Killing form. Thanks to the Milnor result (Lemma \ref{lem: Milnor}) we know that there is only one Lie group structure on SU(3) (as manifold) which may admit a bi-invariant metric. Thus $\omega_\infty$ must be a positive multiple of $\omega_{BF}$, and hence $\omega_\infty = \lambda\,\omega_{BF}$.
    \end{proof}

\subsection{Spin$(5)$}\label{subsec: computations Spin}
    We now consider the Hermitian pluriclosed manifolds $\left(\mbox{Spin}(5),J_\pm,\omega_{BF}\right)$ where $J_\pm$ are the only two isotropic left-invariant complex structure on Spin$(5)$ as described in Section \ref{subsec: isotrop cplx structure Sp} and $\omega_{BF}$ represents the Hermitian metric coming from the Killing form, which is 
    \begin{equation}\label{eq: kill metric Sp}
        \omega_{BF} := \frac{\im}{2} \sum_{k=1}^5 \varphi^k\wedge\bar\varphi^k \;.
    \end{equation}
    
    By computing the complex structure equations, we obtain
    $$\left\{\begin{array}{rcl}
            \p \varphi^1 &=& 0 \\
            \p \varphi^2 &=& \im\varphi^{12} - \varphi^{35} -\im\varphi^{45} \\
            \p \varphi^3 &=& \im\varphi^{13} \pm \im \varphi^{14} +\varphi^{25}\\
            \p \varphi^4 &=& \mp\im\varphi^{13} + \im \varphi^{14} + \im\varphi^{25}\\
            \p \varphi^5 &=& \mp\varphi^{15} 
           \end{array}\right.
     \;\text{ and }\; 
     \left\{\begin{array}{rcl}
            \bar\p \varphi^1 &=& \im\varphi^{2\bar2} + \im\varphi^{3\bar3} \pm\im\varphi^{3\bar4} \mp\im\varphi^{4\bar3} + \im\varphi^{4\bar4} \pm\varphi^{5\bar5} \\
            \bar\p \varphi^2 &=& -\im\varphi^{2\bar1} -\varphi^{3\bar5} +\im\varphi^{4\bar5} \\
            \bar\p \varphi^3 &=& \varphi^{2\bar5} - \im\varphi^{3\bar1} \pm\im\varphi^{4\bar1} \\
            \bar\p \varphi^4 &=& -\im\varphi^{2\bar5} \mp\im\varphi^{3\bar1} -\im\varphi^{4\bar1} \\
            \bar\p \varphi^5 &=& -\varphi^{2\bar3} +\im\varphi^{2\bar4} + \varphi^{3\bar2} -\im\varphi^{4\bar2} \mp\varphi^{5\bar1} 
           \end{array}\right. 
    $$
    where $\pm$ depend on the choice of the complex structure $J_+$ or $J_-$.
    
    We use again the Pittie model for the Dolbeault cohomology of compact simply-connected simple Lie groups of rank two, \cite{Pit}. In particular, when Spin$(5)$ is equipped with a isotropic left-invariant complex structure, (\ref{eq: model cohom}) gives
    \begin{equation*}
        H_{\bar\p}^{\bullet,\bullet}\left(\mbox{Spin}(5)\right) \;\simeq\; \left. \C\left[y_{1,1}\right] \middle\slash \left(\left(y_{1,1}\right)^{4}\right) \right. \otimes \wedge^{\bullet,\bullet} \left( \C\left\langle [u_{2,1}] \right\rangle \oplus \C\left\langle [x_{0,1}] \right\rangle \right) \;,
    \end{equation*}
    and we recover the Hodge diamond
     $$
        \begin{array}{ccccccccccc}
        & & & & & 1 & & & & &  \\
        & & & & 0 & & 1 & & & & \\
        & & & 0 & & 1 & & 0 & & & \\
        & & 0 & & 1 & & 1 & & 0 & & \\
        & 0 & & 0 & & 2 & & 0 & & 0 & \\
        0 & & 0 & & 1 & & 1 & & 0 & & 0 \\
        & 0 & & 0 & & 2 & & 0 & & 0 & \\
        & & 0 &  & 1 & & 1 & & 0 & & \\
        & & & 0 & & 1 & & 0 & & & \\
        & & & & 1 & & 0 & & & & \\
        & & & & & 1 & & & & & 
    \end{array} 
    $$
    Similarly to the previous case, we consider the sub-complex of left-invariant forms
    $$ \iota \colon \bigwedge \left\langle \varphi^1,\, \varphi^2,\, \varphi^3,\, \varphi^4,\, \varphi^5, \, \bar\varphi^1,\, \bar\varphi^2,\, \bar\varphi^3,\, \bar\varphi^4,\,\bar\varphi^5 \right\rangle \hookrightarrow \wedge^{\bullet,\bullet}\mbox{Spin}(5) \;, $$
    and we check that $H_{\bar\p}(\iota)$ is surjective in bi-degree $(0,1)$, $(1,1)$ and $(2,1)$ for both $J_+$ and $J_-$. More precisely, we verify that the sub-complex has cohomologies $H^{0,1}_{\bar\p}(\mbox{Spin}(5))_{inv}$, $H^{1,1}_{\bar\p}(\mbox{Spin}(5))_{inv}$ and $H^{2,1}_{\bar\p}(\mbox{Spin}(5))_{inv}$ of dimension one generated respectively by $$\left[\bar\varphi^1\right],\left[\varphi^{2\bar2} + \varphi^{3\bar3} + \varphi^{4\bar4}\right], \left[\varphi^{12\bar2} + \varphi^{13\bar3} \mp\varphi^{13\bar4} \pm\varphi^{14\bar3} + \varphi^{14\bar4} \pm\im\varphi^{15\bar5} -\im\varphi^{25\bar3} +\varphi^{25\bar4} +\im\varphi^{35\bar2} - \varphi^{45\bar2}\right]\;.$$
    
    Therefore, $H_{\bar\p}(\iota)$ is an isomorphism and we get that the Dolbeaut cohomology ring is
    \begin{eqnarray*}
        H^{\bullet,\bullet}_{\bar\p}(\mbox{Spin}(5)) &=& 
        \C\left\langle 1 \right\rangle 
        \oplus \C\left\langle \left[\varphi^{\bar{1}}\right] \right\rangle 
        \oplus \C\left\langle \left[ \varphi^{2\bar2}+\varphi^{3\bar3}+\varphi^{4\bar4} \right] \right\rangle \oplus\, \C\left\langle \left[\varphi^{2\bar1\bar2}+\varphi^{3\bar1\bar3}+\varphi^{4\bar1\bar4}\right] \right\rangle \\
        && \oplus\, \C\left\langle \left[\varphi^{12\bar2} + \varphi^{13\bar3} \mp\varphi^{13\bar4} \pm\varphi^{14\bar3} + \varphi^{14\bar4} \pm\im\varphi^{15\bar5} -\im\varphi^{25\bar3} +\varphi^{25\bar4} +\im\varphi^{35\bar2} - \varphi^{45\bar2} \right] \right\rangle 
         \\
        && \oplus \C\left\langle \left[ \varphi^{12\bar1\bar2} + \varphi^{13\bar1\bar3} \mp\varphi^{13\bar1\bar4} \pm\varphi^{14\bar1\bar3} + \varphi^{14\bar1\bar4} \pm\im\varphi^{15\bar1\bar5} -\im\varphi^{25\bar1\bar3} +\varphi^{25\bar1\bar4} +\im\varphi^{35\bar1\bar2} \right.\right.\\
        && \left.\left. - \varphi^{45\bar1\bar2} \right],\,\left[ \varphi^{23\bar2\bar3} + \varphi^{24\bar2\bar4} + \varphi^{34\bar3\bar4} \right] \right\rangle \oplus \C\left\langle \left[ 2\varphi^{123\bar{23}} \mp\varphi^{123\bar{24}} \pm \varphi^{124\bar{23}} + 2\varphi^{124\bar{24}} \right.\right.\\
        && \left.\left. \pm \varphi^{125\bar{25}} + 2\varphi^{134\bar{34}} \pm \varphi^{135\bar{35}} \pm \varphi^{145\bar{45}} + \varphi^{235\bar{34}} + \im\varphi^{245\bar{34}} - \varphi^{345\bar{23}} - \im \varphi^{345\bar{24}}
        \right] \right\rangle 
         \\
        && \oplus \C\left\langle \left[ \varphi^{23\bar1\bar2\bar3} + \varphi^{24\bar1\bar2\bar4} + \varphi^{34\bar1\bar3\bar4} \right] \right\rangle \oplus \C\left\langle \left[ \varphi^{234\bar{234}} \right]\right\rangle\\
        && \oplus \C\left\langle \left[ 3\im\varphi^{1234\bar{234}} \mp \varphi^{1235\bar{235}} \mp \varphi^{1245\bar{245}} \mp \varphi^{1345\bar{345}} \right] \right\rangle \oplus \C\left\langle \left[ \varphi^{234\bar{1234}} \right]\right\rangle  \\
        && \oplus \C\left\langle \left[ 3\im\varphi^{1234\bar{1234}} \mp \varphi^{1235\bar{1235}} \mp \varphi^{1245\bar{1245}} \mp \varphi^{1345\bar{1345}} \right] \right\rangle \\
        && \oplus \C\left\langle \left[\varphi^{12345\bar{2345}}\right] \right\rangle \oplus \C\left\langle \left[\varphi^{12345\bar1\bar2\bar3\bar4\bar5}\right] \right\rangle \;,
    \end{eqnarray*}

    As before, applying Theorem 1.3 and Proposition 2.2 of \cite{AK17} we obtain that $H_{BC}(\iota)$ is an isomorphism. In particular, the Bott--Chern cohomology arises from just the left-invariant forms as well as the Aeppli cohomology. A direct computation shows that $\dim H^{1,1}_{A}\left(\mbox{Spin}(5)\right) \simeq \C$ and it is generated by the class of $\omega_{BF}$, that is
    $$ H^{1,1}_A\left(\mbox{Spin}(5)\right) = \C\left<\left[\varphi^{1\bar1}+\varphi^{2\bar2}+\varphi^{3\bar3}+\varphi^{4\bar4}+\varphi^{5\bar5}\right]\right> = \C \left<\left[ \omega_{BF} \right]\right>\;, $$
    The arguments of Theorem \ref{thm: mean theorem1} hold also in this case. In particular, Theorem \ref{thm: GJS convergence} applies ensuring the global stability of the pluriclosed flow on $\left(\mbox{Spin}(5),J_\pm\right)$.
    \begin{thm}\label{thm: mean theorem2}
        Consider $\left(\mbox{Spin}(5),\omega_{BF},J_{\pm}\right)$ where $\omega_{BF}$ is the metric given by the Killing form (as in (\ref{eq: kill metric Sp})) and $J_\pm$ are the isotropic left-invariant complex structure defined in Section \ref{subsec: isotrop cplx structure Sp}. Given any pluriclosed metric $\omega_0$ on $\left(\mbox{Spin}(5),J_\pm\right)$ the solution to the pluriclosed flow with initial data $\omega_0$ exists on $[0, \infty)$ and converges to a Bismut flat metric $\omega_\infty$. In particular, there exists a positive $\lambda$ such that $[\omega_0] = \lambda [\omega_{BF}]$ in $H^{1,1}_A \left(\mbox{Spin}(5)\right)$ and $\omega_\infty = \lambda\,\omega_{BF}$. 
    \end{thm}
    
    For completeness, we compute the whole Bott--Chern cohomology ring, which is
    \begin{eqnarray*}
        H^{\bullet,\bullet}_{BC}(\mbox{Spin}(5)) &=& \C\left\langle 1 \right\rangle \oplus \C\left\langle \left[ \varphi^{2\bar2} + \varphi^{3\bar3} + \varphi^{4\bar4} \right], \left[ \varphi^{3\bar4} - \varphi^{4\bar3} - \im\varphi^{5\bar5} \right] \right\rangle \\
        && \oplus \C\left\langle \left[
        \varphi^{12\bar2} + \varphi^{13\bar3} \mp \varphi^{13\bar4} \pm \varphi^{14\bar3} + \varphi^{14\bar4} \pm \im\varphi^{15\bar5} - \im\varphi^{25\bar3} + \varphi^{25\bar4} + \im\varphi^{35\bar2} - \varphi^{45\bar2}
        \right] \right\rangle\\
        && \oplus \C\left\langle \left[
        \varphi^{2\bar{12}} + \varphi^{3\bar1\bar3} \mp \varphi^{4\bar1\bar3} \pm \varphi^{3\bar1\bar4} + \varphi^{4\bar1\bar4} \mp \im\varphi^{5\bar1\bar5} + \im\varphi^{3\bar2\bar5} + \varphi^{4\bar2\bar5} - \im\varphi^{2\bar3\bar5} - \varphi^{2\bar4\bar5}
        \right] \right\rangle\\
        &&\oplus\, \C\left\langle \left[ 9\varphi^{23\bar{23}} + 9\varphi^{24\bar{24}} + (\im \mp 1)\varphi^{25\bar{13}} - (1 \pm\im)\varphi^{25\bar{14}} + 16\varphi^{34\bar{34}} + (\im \pm 1)\varphi^{35\bar{12}} \right.\right. \\
        &&\left.\left. - 9\im\varphi^{35\bar{45}} - (1 \mp\im)\varphi^{45\bar{12}} +9\im\varphi^{45\bar{35}} \right], \left[ 3\varphi^{23\bar{24}} - 3\varphi^{24\bar{23}} - (1\mp\im)\varphi^{25\bar{13}} -6\im\varphi^{25\bar{25}}  \right.\right.\\
        &&\left.\left. - (\im \pm 1)\varphi^{25\bar{14}} \mp 2\varphi^{34\bar{34}} + (1 \pm\im)\varphi^{35\bar{12}} - 3\im\varphi^{35\bar{35}} + (\im \mp 1)\varphi^{45\bar{12}} - 3\im\varphi^{45\bar{45}}
        \right] \right\rangle \\
        && \oplus \C\left\langle \left[ 
        316\varphi^{123\bar{23}} \mp 332\varphi^{123\bar{24}} \pm 332\varphi^{124\bar{23}} + 316\varphi^{124\bar{24}} - (42\im \mp 50)\varphi^{125\bar{13}} \right.\right.\\
        &&\left.\left. + (42 \pm 50\im)\varphi^{125\bar{14}} \pm 342\im\varphi^{125\bar{25}} + 788\varphi^{134\bar{34}} - (42\im \pm 50)\varphi^{135\bar{12}} \pm 324\im\varphi^{135\bar{35}} \right.\right.\\
        &&\left.\left. - 380 \im\varphi^{135\bar{45}} + (42 \mp 50\im)\varphi^{145\bar{12}} + (189 \mp 125\im)\varphi^{235\bar{34}} + (189\im \pm 125)\varphi^{245\bar{34}} \right.\right.\\
        &&\left.\left. + 380\im\varphi^{145\bar{35}} \pm 324\im\varphi^{145\bar{45}} - (189 \pm 125\im)\varphi^{345\bar{23}} - (189\im \mp 125)\varphi^{345\bar{24}}
        \right] \right\rangle   \\
        && \oplus \C\left\langle \left[ 
        316\varphi^{23\bar{123}} \mp 332\varphi^{24\bar{123}} \pm 332\varphi^{23\bar{124}} + 316\varphi^{24\bar{124}} + (42\im \pm 50)\varphi^{13\bar{125}} \right.\right.\\
        &&\left.\left. + (42 \mp 50\im)\varphi^{14\bar{125}} \mp 342\im\varphi^{25\bar{125}} + 788\varphi^{34\bar{134}} + (42\im \mp 50)\varphi^{12\bar{135}} \mp 324\im\varphi^{35\bar{135}} \right.\right.\\
        &&\left.\left. + 380 \im\varphi^{45\bar{135}} + (42 \pm 50\im)\varphi^{12\bar{145}} + (189 \pm 125\im)\varphi^{34\bar{235}} - (189\im \mp 125)\varphi^{34\bar{245}}  \right.\right.\\
        &&\left.\left.  - 380\im\varphi^{35\bar{145}} \mp 324\im\varphi^{45\bar{145}} - (189 \mp 125\im)\varphi^{23\bar{345}} + (189\im \pm 125)\varphi^{24\bar{345}}
        \right] \right\rangle   \\
        &&\oplus\, \C\left\langle \left[ 
        22\varphi^{123\bar{123}} - 11(1 \pm\im)\varphi^{123\bar{345}} + 22\varphi^{124\bar{124}} + 11(\im \mp 1)\varphi^{124\bar{345}} + 22\im\varphi^{135\bar{145}}  \right.\right.\\
        && \left.\left.  + 11(1 \mp\im)\varphi^{134\bar{235}} - 11(\im \pm 1)\varphi^{134\bar{245}} - 22\im\varphi^{145\bar{135}} \pm 208\im\varphi^{234\bar{234}} + 276\varphi^{235\bar{235}} \right.\right.\\
        && \left.\left. + (209 \pm\im)\varphi^{235\bar{134}} \pm 208\varphi^{235\bar{245}} + (209\im \mp 1)\varphi^{245\bar{134}} \mp 208\varphi^{245\bar{235}} + 276\varphi^{245\bar{245}} \right.\right.\\
        && \left.\left. - (209 \mp\im)\varphi^{345\bar{123}} - (209\im \pm 1)\varphi^{345\bar{124}} + 276\varphi^{345\bar{345}}
        \right], \left[ 22\varphi^{123\bar{124}} - 22\varphi^{124\bar{123}}\right.\right.\\
        && \left.\left.  - 11(\im \mp 1)\varphi^{123\bar{345}}  - 11(1 \pm\im)\varphi^{124\bar{345}} - 11(\im \pm 1)\varphi^{134\bar{235}} - 11(1 \mp\im)\varphi^{134\bar{245}} \right.\right.\\
        && \left.\left. + 44\im\varphi^{125\bar{125}} - 22\im\varphi^{135\bar{135}} - 22\im\varphi^{145\bar{145}} + 452\im\varphi^{234\bar{234}} - (\im \mp 341)\varphi^{235\bar{134}} \right.\right.\\
        && \left.\left. \pm 340\varphi^{235\bar{235}} + 452\varphi^{235\bar{245}} + (1 \pm 341\im)\varphi^{245\bar{134}}  - 452\varphi^{245\bar{235}} \pm 340\varphi^{245\bar{245}} \right.\right.\\
        && \left.\left. - (\im \pm 341)\varphi^{345\bar{123}} + (1 \mp 341\im)\varphi^{345\bar{124}} \pm 340\varphi^{345\bar{345}} \right] \right\rangle \\
        && \oplus \C\left\langle \left[ 
        7\varphi^{1234\bar{234}} + 2\varphi^{1235\bar{134}} \pm 7\im\varphi^{1235\bar{235}} - 7\im\varphi^{1235\bar{245}} + 2\im\varphi^{1245\bar{134}} \right.\right.\\
        && \left.\left. + 7\im\varphi^{1245\bar{235}}  \pm 7\im\varphi^{1245\bar{245}} + 2\varphi^{1345\bar{123}} + 2\im\varphi^{1345\bar{124}} \pm 7\im\varphi^{1345\bar{345}}
        \right] \right\rangle \\
        && \oplus \C\left\langle \left[ 
        7\varphi^{234\bar{1234}} + 2\varphi^{134\bar{1235}} \mp 7\im\varphi^{235\bar{1235}} + 7\im\varphi^{245\bar{1235}} - 2\im\varphi^{134\bar{1245}} \right.\right.\\
        && \left.\left. - 7\im\varphi^{235\bar{1245}} \mp 7\im\varphi^{245\bar{1245}} + 2\varphi^{123\bar{1345}} - 2\im\varphi^{124\bar{1345}} \mp 7\im\varphi^{345\bar{1345}}
        \right] \right\rangle \\
        && \oplus \C\left\langle \left[ 
        2\varphi^{1234\bar{1234}} + \im\varphi^{1235\bar{1245}} - \im\varphi^{1245\bar{1235}} + 3\varphi^{2345\bar{2345}}
        \right] \right\rangle \\
        && \oplus \C\left\langle \left[ \varphi^{12345\bar{2345}} \right] \right\rangle \oplus \C\left\langle \left[ \varphi^{2345\bar{12345}} \right] \right\rangle \oplus \C\left\langle \left[ \varphi^{12345\bar{12345}} \right] \right\rangle \;.
    \end{eqnarray*}
    Hence, the Bott-Chern numbers are
    $$
        \begin{array}{ccccccccccc}
        & & & & & 1 & & & & & \\
        & & & & 0 & & 0 & & & & \\
        & & & 0 & & 2 & & 0 & & & \\
        & & 0 & & 1 & & 1 & & 0 & & \\
        & 0 & & 0 & & 2 & & 0 & & 0 & \\
        0 & & 0 & & 1 & & 1 & & 0 & & 0 \\
        & 0 & & 0 & & 2 & & 0 & & 0 & \\
        & & 0 &  & 1 & & 1 & & 0 & & \\
        & & & 0 & & 1 & & 0 & & & \\
        & & & & 1 & & 1 & & & & \\
        & & & & & 1 & & & & & 
    \end{array} 
    $$
    
\subsection{G$_2$}\label{subsec: computations G2}
     Finally, we consider the Hermitian pluriclosed manifolds $\left(\mbox{G}_2,J_{\pm},\omega_{BF}\right)$ where $J_{\pm}$ are the only two isotropic left-invariant complex structure on G$_2$ as described in the Section \ref{subsec: isotrop cplx structure G2} and $\omega_{BF}$ represents the Hermitian metric coming from the Killing form, which is 
    \begin{equation}\label{eq: kill metric G2}
        \omega_{BF} := \im \left( 3\varphi^{1\bar1} + \varphi^{2\bar2} - 3\varphi^{3\bar3} + 12\varphi^{4\bar4} - 36\varphi^{5\bar5} + 36\varphi^{6\bar6} - 12\varphi^{7\bar7}  \right) \;.
    \end{equation}
    By computing the complex structure equations, we obtain
    $$\left\{\begin{array}{rcl}
            \p \varphi^1 &=& -2\varphi^{17} \\
            \p \varphi^2 &=& (\pm \im\sqrt{3} + 3)\varphi^{27} \\
            \p \varphi^3 &=&  \varphi^{12} + (\pm \im\sqrt{3} + 1)\varphi^{37}\\
            \p \varphi^4 &=&  \varphi^{13} + (\pm \im\sqrt{3} - 1)\varphi^{47}\\
            \p \varphi^5 &=&  \varphi^{14} + (\pm \im\sqrt{3} - 3)\varphi^{57}\\
            \p \varphi^6 &=&  \varphi^{25} - \varphi^{34} \pm 2\im\sqrt{3}\varphi^{67}\\
            \p \varphi^7 &=& 0
           \end{array}\right.
     \;\text{ and }\; 
     \left\{\begin{array}{rcl}
            \bar\p \varphi^1 &=& -\varphi^{2\bar3} + 4\varphi^{3\bar4} - 12\varphi^{4\bar5} - 2\varphi^{7\bar1}\\
            \bar\p \varphi^2 &=&  (\pm \im\sqrt{3} - 3)\varphi^{2\bar7} - 3\varphi^{3\bar1} - 36\varphi^{6\bar5} \\
            \bar\p \varphi^3 &=& (\pm \im\sqrt{3} - 1)\varphi^{3\bar7} - 4\varphi^{4\bar1} - 12\varphi^{6\bar4} \\
            \bar\p \varphi^4 &=& (\pm \im\sqrt{3} + 1)\varphi^{4\bar7} - 3\varphi^{5\bar1} - 3\varphi^{6\bar3}\\
            \bar\p \varphi^5 &=& (\pm \im\sqrt{3} + 3)\varphi^{5\bar7} - \varphi^{6\bar2}   \\
            \bar\p \varphi^6 &=& \pm 2\im\sqrt{3}\varphi^{6\bar7}  \\
            \bar\p \varphi^7 &=& 1/2\varphi^{1\bar1} + (\pm 1/12\im\sqrt{3} - 1/4)\varphi^{2\bar2} \\
            && + (\mp 1/4\im\sqrt{3} + 1/4)\varphi^{3\bar3} + (\pm \im\sqrt{3} + 1)\varphi^{4\bar4} \\
            && + (\mp3\im\sqrt{3} - 9)\varphi^{5\bar5} \pm 6\im\sqrt{3}\varphi^{6\bar6}
           \end{array}\right. 
    $$
    As for the previous cases, we use Pittie's model for the Dolbeaut cohomology of compact simply-connected simple Lie groups of rank two, \cite{Pit}. In particular, when G$_2$ is equipped with a isotropic left-invariant complex structure, (\ref{eq: model cohom}) gives
    \begin{equation*}
        H_{\bar\p}^{\bullet,\bullet}(\mbox{G}_2) \;\simeq\; \left. \C\left[y_{1,1}\right] \middle\slash \left(\left(y_{1,1}\right)^{6}\right) \right. \otimes \wedge^{\bullet,\bullet} \left( \C\left\langle [u_{2,1}] \right\rangle \oplus \C\left\langle [x_{0,1}] \right\rangle \right) \;,
    \end{equation*}
    and we recover the Hodge numbers
    $$
        \begin{array}{ccccccccccccccc}
          &   &   &   &   &   &   & 1 &   &   &   &   &   &   &  \\
          &   &   &   &   &   & 0 &   & 1 &   &   &   &   &   &  \\
          &   &   &   &   & 0 &   & 1 &   & 0 &   &   &   &   &  \\
          &   &   &   & 0 &   & 1 &   & 1 &   & 0 &   &   &   &  \\
          &   &   & 0 &   & 0 &   & 2 &   & 0 &   & 0 &   &   &  \\
          &   & 0 &   & 0 &   & 1 &   & 1 &   & 0 &   & 0 &   &  \\
          & 0 &   & 0 &   & 0 &   & 2 &   & 0 &   & 0 &   & 0 &  \\
        0 &   & 0 &   & 0 &   & 1 &   & 1 &   & 0 &   & 0 &   & 0\\
          & 0 &   & 0 &   & 0 &   & 2 &   & 0 &   & 0 &   & 0 &  \\
          &   & 0 &   & 0 &   & 1 &   & 1 &   & 0 &   & 0 &   &  \\
          &   &   & 0 &   & 0 &   & 2 &   & 0 &   & 0 &   &   &  \\
          &   &   &   & 0 &   & 1 &   & 1 &   & 0 &   &   &   &  \\
          &   &   &   &   & 0 &   & 1 &   & 0 &   &   &   &   &  \\
          &   &   &   &   &   & 1 &   & 0 &   &   &   &   &   &  \\
          &   &   &   &   &   &   & 1 &   &   &   &   &   &   &  
    \end{array} 
    $$
    As before, we consider the sub-complex of left-invariant forms
    $$ \iota \colon \bigwedge \left\langle \varphi^1,\, \varphi^2,\, \varphi^3,\, \varphi^4,\, \varphi^5, \, \varphi^6,\, \varphi^7, \, \bar\varphi^1,\, \bar\varphi^2,\, \bar\varphi^3,\, \bar\varphi^4,\,\bar\varphi^5,\,\bar\varphi^6,\,\bar\varphi^7 \right\rangle \hookrightarrow \wedge^{\bullet,\bullet}\mbox{G}_2 \;, $$
    and we check that the sub-complex has cohomologies $H^{0,1}_{\bar\p}(\mbox{G}_2)_{inv}$, $H^{1,1}_{\bar\p}(\mbox{G}_2)_{inv}$ and $H^{2,1}_{\bar\p}(\mbox{G}_2)_{inv}$ of dimension one, for both the complex structures $J_+$ and $J_-$.
    
    Therefore, $H_{\bar\p}(\iota)$ is an isomorphism and Theorem 1.3 and Proposition 2.2 of \cite{AK17} apply giving that also $H_{BC}(\iota)$ is an isomorphism. In particular, the Bott--Chern and the Aeppli cohomologies arise from just the left-invariant forms. A direct computation shows that $\dim H^{1,1}_{A}(\mbox{G}_2) \simeq \C$ and it is generated by the class of $\omega_{BF}$, that is
    $$ H^{1,1}_A(\mbox{G}_2) = \C\left<\left[3\varphi^{1\bar1} + \varphi^{2\bar2} - 3\varphi^{3\bar3} + 12\varphi^{4\bar4} - 36\varphi^{5\bar5} + 36\varphi^{6\bar6} - 12\varphi^{7\bar7}\right]\right> = \C \left<\left[ \omega_{BF} \right]\right>\;, $$
    Thus Theorem \ref{thm: GJS convergence} applies also in this case ensuring the global stability of the pluriclosed flow on $\left(\mbox{G}_2,J_\pm\right)$. 
    In particular, using the same arguments of Theorem \ref{thm: mean theorem1} we can prove the following result.
    \begin{thm}\label{thm: mean theorem3}
        Consider $\left(\mbox{G}_2,\omega_{BF},J_{\pm}\right)$ where $\omega_{BF}$ is the metric given by the Killing form (as in (\ref{eq: kill metric G2})) and $J_{\pm}$ are the isotropic left-invariant complex structure defined in Section \ref{subsec: isotrop cplx structure G2}. Given any pluriclosed metric $\omega_0$ on $\left(\mbox{G}_2,J_{\pm}\right)$ the solution to the pluriclosed flow with initial data $\omega_0$ exists on $[0, \infty)$ and converges to a Bismut flat metric $\omega_\infty$. In particular, there exists a positive $\lambda$ such that $[\omega_0] = \lambda [\omega_{BF}]$ in $H^{1,1}_A \left(\mbox{G}_2\right)$ and $\omega_\infty = \lambda\,\omega_{BF}$. 
    \end{thm}
    
\section*{Conflict of Interest}
The author declares that he has no conflict of interest.


\begin{thebibliography}{10}
\bibitem{aeppli}
Aeppli A., On the cohomology structure of Stein manifolds, {\em Proc. Conf. Complex Analysis (Minneapolis, Minn., 1964)}, Springer, Berlin, 1965, pp. 58--70.

\bibitem {AI} 
Alexandrov B. and Ivanov S.,  Vanishing theorems on Hermitian manifolds, {\em Diff. Geom. Appl.} {\bf 14}, (2001) 251-265.

\bibitem{AK17}
Angella D. and Kasuya H., Bott--Chern cohomology of solvmanifolds, {\em Annals of Global Analysis and Geometry} \textbf{52}, n°4, (2017) 363--411.

\bibitem{AOUV}
Angella D., Otal A., Ugarte L. and
Villacampa R., On Gauduchon connections with K{\"a}hler-like curvature, To appear in {\em Communications in Analysis and Geometry}.

\bibitem{AT15}
Angella D. and Tomassini A., On Bott--Chern cohomology and formality, {\em J. Geom. Phys.} {\bf 93}, (2015) 52–61.

\bibitem{Me}
Barbaro G., On the curvature of the Bismut connection: Bismut Yamabe problem and Calabi-Yau with torsion metrics, {\em arXiv:}2109.06159 (2021).

\bibitem{bott-chern}
Bott R., Chern S.~S., Hermitian vector bundles and the equidistribution of the zeroes of their holomorphic sections, {\em Acta Math.} \textbf{114} (1965), no.~1, 71--112.

\bibitem{GJS}
Garcia--Fernandez M., Jordan J. and Streets J., Non-K{\"a}hler Calabi--Yau geometry and pluriclosed flow, (2021) {\em arXiv:}2106.1371.

\bibitem {Mil} 
Milnor J., Curvatures of left invariant metrics on Lie groups,  {\em Advances in Math.} \textbf{21}, n°3, (1976) 293-329.

\bibitem{Pit}
Pittie H. V., The Dolbeault-cohomology ring of a compact, even-dimensional lie group, {\em Proc. Indian Acad. Sci.} \textbf{98}, n° 2-3, (1988) 117-152.

\bibitem {Sam} 
Samelson H., A class of complex analytic manifolds, {\em Portugaliae Math.} {\bf 12}, (1953) 129-132.

\bibitem{ST10}
Streets J. and Tian G., A parabolic flow of pluriclosed metrics, {\em International Mathematics Research Notices. IMRN} {\bf 16}, (2010) 3101--3133.

\bibitem{WYZ}
Wang Q., Yang B. and Zheng F., On Bismut flat manifolds, {\em Transactions of the American Mathematical Society} \textbf{373}, n°8, (2020) 5747-5772.


\end{thebibliography}
\end{document}